\documentclass[11pt]{article}
\usepackage[latin1]{inputenc}
\usepackage[T1]{fontenc}
\usepackage{amsmath,amssymb,amstext}
\usepackage{theorem}
\usepackage{a4wide}
\usepackage{multicol}
\usepackage[english]{babel}
\usepackage{enumerate}
\usepackage{eufrak}

\def\R{{ \mathbb{R}}}
\newtheorem{theorem}{Theorem}

\newtheorem{lemma}[theorem]{Lemma}

\newtheorem{remark}[theorem]{Remark}

\newenvironment{proof}[1][Proof]{\textbf{#1.} }{\ \rule{0.5em}{0.5em}}
\def\cal{\mathcal}

\renewcommand{\geq}{\geqslant}
\def\leq{\leqslant}

\newcommand{\N}{\mathbb{N}}

\newcommand{\E}{\mathbb{E}}

\def\cal{\mathcal}

\def\ep{\varepsilon}

\def\r{{\mathbb R}}
\def\0.5{{\frac{1}{2}}}

\newcommand\1{\leavevmode\hbox{\rm \small1\kern-0.35em\normalsize1}}
\newcommand\ind[1]{\1_{\{#1\}}}
\newcommand{\limite}[2]{\xrightarrow[#1]{#2}}

\renewcommand{\thefootnote}{\fnsymbol{footnote}}
\usepackage{mathrsfs}
\begin{document}

\renewcommand{\thefootnote}{\arabic{footnote}}

\begin{center}
{\large{\bf Almost sure central limit theorems for random ratios and applications to LSE for fractional Ornstein-Uhlenbeck processes}}\\~\\
Peggy Cénac\footnote{Institut de Mathématiques de Bourgogne, Université de Bourgogne,
9 Rue Alain Savary, 21078 Dijon, France. E-mail: {\tt  peggy.cenac@u-bourgogne.fr} } and
Khalifa Es-Sebaiy\footnote{National School of Applied Sciences -
Marrakesh, Cadi Ayyad University, Morocco. E-mail: {\tt
k.essebaiy@uca.ma} }
\\
{\it Université de Bourgogne and  Cadi Ayyad University }\\~\\
\end{center}

{\small \noindent {\bf Abstract:} We investigate an almost sure
limit theorem (ASCLT) for sequences of random variables having the
form of a ratio of two terms such that the numerator  satisfies the
ASCLT and the denominator is a positive term which converges almost
surely to $1$.  This result leads to the ASCLT for
 least square estimators for  Ornstein-Uhlenbeck process driven
by fractional Brownian motion.
 }\\

{\small \noindent {\bf Key words}: Almost sure limit theorem;
 least squares estimator;  fractional Ornstein Uhlenbeck process; multiple stochastic integrals.
\\


\section{Introduction}
The Almost Sure Central Limit Theorem (ASCLT) was simultaneously proven by Brosamler
\cite{brosamler} and Schatte \cite{schatte}. The simplest form of the ASCLT (see Lacey and Phillip \cite{LP}) states that if  $\{X_n, n\geq1\}$ is a sequence of real-valued independent identically distributed random
variables with $\E(X_1) = 0$, $\E(X_1^2)= 1$, denoting the normalized
partial sums by $S_n~=~\frac{1}{\sqrt{n}}\left(X_1+\ldots+X_n\right)$
then, almost surely, for all $z\in\R$,
\begin{eqnarray*}
  \frac{1}{\log
n}\sum_{k=1}^{n}\frac{1}{k}\ind{S_k\leq z}\limite{n\to \infty}{a.s.}
P(N\leq z),
\end{eqnarray*}
where $N$ is a $\mathcal{N }(0, 1)$ random variable and $\ind{A}$ denotes the indicator of the set $A$. Equivalently, for any bounded and continuous
function $\varphi:\R\rightarrow\R,$ one has almost surely,
\[\frac{1}{\log n}\sum_{k=1}^{n}\frac{1}{k}\varphi(S_k)\limite{n\to \infty}{a.s.} \E(\varphi(N)).\]
The ASCLT was first stated by Lévy \cite{Levy} without
proof. For more discussions about ASCLT see for example Berkes and Cs\'aki
\cite{BC} and the references in the survey paper Berkes
\cite{berkes}.

Ibragimov and Lifshits \cite{IL2000, IL98} give a criterion (see Theorem~\ref{Ibragimov Lifshits}) for the ASCLT based on the rate of convergence of the empirical characteristic function. Using this
criterion and Malliavin calculus, Bercu \emph{et al.}\@ \cite{BNT} provide a criterion for ASCLT for functionals of
general Gaussian fields.

Our first aim is to prove an almost sure
central limit theorem for a sequence of the form
$\{G_n/R_n\}_{n\geq1}$ where $\{G_n\}_{n\geq1}$
 satisfies the ASCLT and  $\{R_{n}\}_{n\geq
1}$ is a sequence of positive random variables not necessarily independent with $\{G_n\}$ and converging almost
surely to $1$ (see Theorem~\ref{ASCLT quotient}). We apply our ASCLT
to a fractional Ornstein-Uhlenbeck process  $X=\left\{X_t,
t\geq0\right\}$ defined as
\begin{eqnarray*} X_0=0, \mbox{ and }\ dX_t=-\theta X_tdt+dB_t,\quad t\geq0,
\end{eqnarray*} where  $B=\left\{B_t,
t\geq0\right\}$ is a fractional Brownian motion with
Hurst parameter $H\in(\frac{1}{2},1)$, and where $\theta$ is a real parameter. $\theta$ is unknown and estimated with least squares estimators (LSE). Theorem~\ref{ASCLT quotient} leads to the ASCLT for the LSE in this model.
\medskip

\emph{In the continuous case}, recently, the parametric estimation of continuously
observed fractional Ornstein-Uhlenbeck process is
 studied in Hu and Nualart \cite{ HN} and Belfadli \emph{et al.}\@ \cite{BEO}  by using the least
 squares estimator
$\widehat{\theta}_t$ of $\theta$ given by
\begin{eqnarray*} \widehat{\theta}_t =\frac{\int_0^tX_sdX_s}{\int_0^tX_s^2ds},\quad t\geq0.
\end{eqnarray*}
Hu and Nualart \cite{HN} proved the strong consistency and asymptotic normality of
$\widehat{\theta}_t$ in the ergodic case, that is when $\theta > 0$. In the non-ergodic case $\theta<0$,
 Belfadli \emph{et al.}\@ \cite{BEO} established that the LSE $\widehat{\theta}_t $ of $\theta$ is
strongly consistent and asymptotically Cauchy.

 In this paper, we
focus our discussion on the ergodic case $\theta > 0$. We shall prove that when  $H\in(1/2,3/4)$ the sequence $\{\sqrt{n}({\theta
}-\widehat{{\theta }}_{n})\}_{n\geq1}$ satisfies the ASCLT (see
Theorem \ref{ASCLT continuous case}).

\medskip

\emph{In the discrete case}, the process $X$
is observed equidistantly in time with the step size $\Delta_n$, that is for any  $i\in~\{0,\ldots,n\},$ $t_i=i\Delta_{n}$.
 Consider the least squares estimator $\widetilde{\theta}_n$ of $\theta$
defined by
\begin{equation*}
\widetilde{\theta}_{n}=-\frac{\sum_{i=1}^{n}
X_{t_{i-1}}(X_{t_{i}}-X_{t_{i-1}})}{\Delta_{n}\sum_{i=1}^{n}
X_{t_{i-1}}^2}.
\end{equation*}
When $\theta>0$, Es-Sebaiy \cite{khalifa} established the
convergence in probability and gave the rate of convergence of the
least squares estimator $\widetilde{\theta}_{n}$.

In the present work, we shall prove the strong consistency of
$\widetilde{\theta}_{n}$ (the almost sure convergence of
$\widetilde{\theta}_{n}$ to the parameter $\theta$, see
Theorem~\ref{strong consistency}). We shall also prove that, in the
case when $H\in(1/2,3/4)$, the sequence
\[\left\{\frac{\sqrt{n\Delta_n}}{\sigma_n}({\theta
}-\widetilde{{\theta }}_{n})\right\}_{n\geq1}\]
satisfies the ASCLT
 (see Theorem~\ref{ASCLT discrete case}) where $\sigma_n$ is the positive normalizing sequence defined in (\ref{defSigma}).

\medskip

The paper is organized as follows. Section 2 contains the
basic tools of Malliavin calculus for the fractional Brownian motion
needed throughout the paper. In section 3 we prove the ASCLT for a
sequence of random variables having the form of a ratio of two terms
such that the numerator  satisfies the ASCLT and the denominator is a
positive term which converges almost surely to $1$. In Section 4, we
use our ASCLT to study the ASCLT for  least square estimators
 for the fractional Ornstein-Uhlenbeck processes.

\section{Preliminaries}
In this section we describe some basic facts on the  stochastic calculus with respect to a
fractional Brownian motion. For more complete presentation on the subject, see  Nualart \cite{Nualart-book}.\\
The  fractional Brownian motion $\{B_t, t\geq0\}$ with Hurst
parameter $H\in(0,1)$, is defined as a centered Gaussian process
starting from zero with covariance
\[R_H(t,s):=\E(B_tB_s)=\frac{1}{2}\left(t^{2H}+s^{2H}-|t-s|^{2H}\right).\]
Assume that $B$ is defined on a complete probability space
$(\Omega, \mathcal{F}, P)$ such that $\mathcal{F}$ is the
sigma-field generated by $B$. By  Kolmogorov's continuity criterion
and the equality
\[\E\left(B_t-B_s\right)^2=|s-t|^{2H};\ s,\ t\geq~0,\] $B$ has H\"older continuous paths of order $H-\varepsilon$,
 for all $\varepsilon\in(0,H)$.

 Fix a time interval $[0, T]$. We denote by $\cal{H}$ the canonical Hilbert space associated to the  fractional
Brownian motion $B$. That is, $\cal{H}$ is the closure of the linear span $\mathcal{E}$ generated by the indicator functions
 $\ind{[0,t]},\ t\in[0,T] $ with respect to the scalar product \[\langle \ind{[0,t]},\ind{[0,s]}\rangle=R_H(t,s).\]
We denote by $|\cdot|_{\cal H}$ the associated norm. The mapping
$\displaystyle \1_{[0,t]} \mapsto B_{t}$ can be extended to an
isometry between $\mathcal H$ and the Gaussian space associated with
$B$. We denote this isometry by
\begin{equation*}
\varphi\mapsto B(\varphi)=\int_0^T \varphi(s)\, dB_s.
\end{equation*}
When $H>\frac{1}{2}$ the elements of $\cal{H}$ may be not functions but distributions of negative order (see Pipiras and Taqqu \cite{PT}). Therefore, it is of interest to know significant subspaces of functions contained in it.\\
Let $|\cal{H}|$  be the set of measurable functions  $\varphi$   on $[0, T]$ such that
\[\|\varphi\|_{|\cal{H}|}^2:=H(2H-1)\int_0^T\int_0^T|\varphi(u)||\varphi(v)||u-v|^{2H-2}dudv<\infty.\] Note that, if $\varphi,\ \psi\in|\cal{H}|$, then
\[\E\left(B(\varphi)B(\psi)\right)=H(2H-1)\int_0^T\int_0^T\varphi(u)\psi(v)|u-v|^{2H-2}dudv.\]
It follows actually from Pipiras and Taqqu \cite{PT} that the space $|\cal{H}|$ is a Banach space for the norm $\|.\|_{|\cal{H}|}$ and it is included in $\cal{H}$.
Moreover one has
\begin{eqnarray}\label{inclusions} L^2([0, T]) \subset L^{\frac{1}{H}}([0, T])\subset|\cal{H}|\subset\cal{H}.
\end{eqnarray}

Let
$\mathrm{C}_b^{\infty}(\R^n,\R)$
 be the class of infinitely
differentiable functions $f: \R^n \longrightarrow \R$ such that $f$ and all its partial derivatives are bounded.
We denote by $\cal{S}$ the class of  cylindrical random
variables $F$ of the form
\begin{eqnarray}F = f(B(\varphi_1),...,B(\varphi_n)),\label{functional}\end{eqnarray} where $n\geq1$, $f\in \mathrm{C}_b^{\infty}(\R^n,\R)$
 and $\varphi_1,...,\varphi_n\in\cal{H}.$\\
 The derivative operator $D$ of a  cylindrical random variable $F$ of the form (\ref{functional}) is
defined as the $\cal{H}$-valued random variable
$$D_tF=\sum_{i=1}^{N}\frac{\partial f}{\partial x_i}(B(\varphi_1),...,B(\varphi_n))\varphi_i(t).$$
In this way the derivative $DF$ is an element of $L^2(\Omega
;\cal{H})$. For $p\geq1$, let $D^{1,p}$ be the closure of $\mathcal{S}$ with
respect to the norm defined by
$$\|F\|_{1,p}^p=\E(\|F\|^p)+\E(\|DF\|^p_{{\cal{H}}}).$$
The divergence operator $\delta$ is the adjoint of the derivative
operator $D$. Concretely, a random variable $u\in
L^2(\Omega;\cal{H})$ belongs to the domain of the divergence
operator $Dom(\delta)$ if for every $F\in \mathcal{S}$,
\[\E\left|\langle DF,u\rangle_{\cal{H}}\right|\leq c\|F\|_{L^2(\Omega)}.\]
In this case $\delta(u)$ is given by the duality relationship
\begin{eqnarray*}\E(F\delta(u))=\E\left<DF,u\right>_{\cal{H}}
\end{eqnarray*}
for any $F\in D^{1,2}$. We will
make use of the notation
$$\delta(u)=\int_0^Tu_s d B_s,\quad u\in Dom(\delta).$$
In particular, for $h\in\cal{H}$, $B(h)=\delta(h)=\int_0^Th_s d B_s.$

Assume that $H\in(\frac{1}{2},1)$. If $p\geq1$ and $u\in
D^{1,p}(|{\cal{H}}|)$  then $u$ belongs to $Dom(\delta)$ and we have
(see Nualart \cite[Page 292]{Nualart-book})
\begin{eqnarray*}\E(|\delta(u)|^p)\leq
c\left(\|\E(u)\|_{|{\cal{H}}|}^p+\E\left(\|Du\|_{|{\cal{H}}|\otimes|{\cal{H}}|}^p
\right)\right),
\end{eqnarray*}where the constant $c $ depends only on $p$ and $H$.\\
As a consequence, applying (\ref{inclusions}) it comes
\begin{eqnarray}\label{majoration second moment of skorohod}
\E(|\delta(u)|^p)\leq c\left(\|\E(u)\|_{L^{\frac{1}{H}}([0,
T])}^p+\E\left(\|Du\|_{L^{\frac{1}{H}}([0, T]^2)}^p \right)\right).
\end{eqnarray}

For every $n\geq1$, let ${\cal{H}}_n$ be the nth Wiener chaos of $B$, that is, the closed linear subspace of $L^2(\Omega)$ generated by the random variables $\{H_n(B(h)), h\in{\cal{H}}, \|h\|_{{\cal{H}}} = 1\}$ where $H_n$
is the nth Hermite polynomial. The mapping $I_n(h^{\otimes n})=n!H_n(B(h))$ provides a linear isometry between the symmetric tensor product ${\cal{H}}^{\odot n}$
(equipped with the modified norm $\|.\|_{{\cal{H}}^{\odot n}}=\frac{1}{\sqrt{n!}}\|.\|_{{\cal{H}}^{\otimes n}}$) and ${\cal{H}}_n$. For every  $f,g\in{{\cal{H}}}^{\odot n}$
the following multiplication formula holds:
\[\E\left(I_n(f)I_n(g)\right)=n!\langle f,g\rangle_{{\cal{H}}^{\otimes n}}.\]
 On the other hand, it is well-known   that $L^2(\Omega)$ can be decomposed into the infinite orthogonal sum of the spaces ${\cal{H}}_n$. That is, any square
integrable random variable $F\in L^2(\Omega)$  admits the following chaotic expansion
\[F=\E(F)+\sum_{n=1}^{\infty}I_n(f_n),\]
where the $f_n \in{{\cal{H}}}^{\odot n}$  are uniquely determined by
$F$.\\
Let $\{e_n, n\geq 1\}$ be a complete orthonormal system in
${\cal{H}}$. Given $f\in{\cal{H}}^{\odot p}$ and
$g\in{\cal{H}}^{\odot q}$, for every $r=0, \dots, p\wedge q$, the
$r$th contraction of $f$ and $g$ is the element of
${\cal{H}}^{\otimes(p+q-2r)}$ defined as
$$
f\otimes_r g=\sum_{i_1=1, \dots, i_r=1}^{\infty} \langle
f,e_{i_1}\otimes\cdots\otimes e_{i_r}\rangle_{{\cal{H}}^{\otimes r}}
\otimes \langle g,e_{i_1}\otimes\cdots\otimes
e_{i_r}\rangle_{{\cal{H}}^{\otimes r}}.
$$
In particular, note that $f\otimes_0g=f\otimes g$ and when $p=q$,
that $f\otimes_pg=\langle f, g\rangle_{{\cal{H}}^{\otimes p}}$.
 Since, in
general, the contraction $f\otimes_rg$ is not necessarily symmetric,
we denote its symmetrization by $f\widetilde\otimes_rg \in
{\cal{H}}^{\odot(p+q-2r)}$. When $f\in {\cal{H}}^{\odot q}$, we
write $I_q(f)$ to indicate its $q$th multiple integral with respect
to $X$. The following formula is useful to compute the product of
such multiple integrals: if $f\in {\cal{H}}^{\odot p}$ and $g\in
{\cal{H}}^{\odot q}$, then
\begin{equation}
\label{eq:multiplication} I_p(f)I_q(g)=\sum_{r=0}^{p\wedge q} r!
\left(\!\!\begin{array}{c}p\\r\end{array}\!\!\right)
\left(\!\!\begin{array}{c}q\\r\end{array}\!\!\right)
I_{p+q-2r}(f\widetilde\otimes_rg).
\end{equation}
Let us now recall the criterion of Ibragimov and Lifshits \cite{IL2000} which plays  a crucial role in Bercu \emph{et al.}\@ \cite{BNT} to study ASCLTs for sequences of functionals
of general Gaussian fields.
\begin{theorem}[Ibragimov and Lifshits \cite{IL2000}]\label{Ibragimov
Lifshits} Let $\{G_n\}$ be a sequence of random variables converging
in distribution towards a random variable $G_\infty$, and set
\[
\Delta_n(t)=\frac1{\log n}\sum_{k=1}^n \frac{1}k
\big(e^{itG_k}-\E(e^{itG_\infty})\big).
\]
Assuming for all $r>0$
\begin{equation*}
\sup_{|t|\leq r}\sum_n \frac{E\vert \Delta_n(t)\vert^2}{n\log
n}<\infty,
\end{equation*}
then, almost surely, for all continuous and bounded functions
$\varphi:\R\to\R$, one has
\[
\frac{1}{\log n}\sum_{k=1}^{n} \frac{1}{k}\,\varphi(G_k)
\limite{n\to \infty}{a.s.} \E(\varphi(G_\infty)).
\]
\end{theorem}

For the rest of the paper, we will use the standard notation
$\phi(z):= P(N \leq z)$ where $N$ is a $\mathcal{N }(0, 1)$ random
variable. We will denote by $C(\theta,H)$ a generic positive
constant which depends only on $\theta$ and $H$.

\section{Almost sure central limit theorems}

In this section we state and  prove our results concerning the ASCLT for the
sequences of $\R$-valued random variables of the form
$\{G_n/R_n\}_{n\geq1}$ and $\{G_n+R_n\}_{n\geq1}$.
\begin{theorem}\label{ASCLT quotient}Let $\{G_{n}\}_{n\geq 1}$ be a sequence of $\R$-valued random variables
satisfying the ASCLT.  Let $\{R_{n}\}_{n\geq 1}$ be a sequence of
positive random variables converging almost surely to $1$. Then
$\{G_n/R_n\}_{n\geq1}$ satisfies the ASCLT. In other words, if $N$
is a $\mathcal{N }(0, 1)$ random variable, then, almost surely, for
all $z \in\R$,
\begin{eqnarray*}
  \frac{1}{\log
n}\sum_{k=1}^{n}\frac{1}{k}\ind{G_k\leq
zR_k}\limite{n \to \infty}{a.s.} \phi(z).
 \end{eqnarray*}
 \end{theorem}
 \begin{theorem}\label{ASCLT sum}Let $\{G_{n}\}_{n\geq 1}$ be a sequence of $\R$-valued random variables
satisfying the ASCLT.  Let $\{R_{n}\}_{n\geq 1}$ be a sequence of
$\R$-valued random variables converging almost surely to $0$. Then
$\{G_n+R_n\}_{n\geq1}$ satisfies the ASCLT. In other words, almost
surely, for all $z \in\R$,
\begin{eqnarray*}
  \frac{1}{\log
n}\sum_{k=1}^{n}\frac{1}{k}\ind{G_k+R_k\leq
z}\limite{n \to \infty}{a.s.} \phi(z).
 \end{eqnarray*}
\end{theorem}
\begin{remark} A similar result to Theorem~\ref{ASCLT sum} for the ASCLT of $\{G_n+R_n\}_{n\geq1}$ where
$\{R_{n}\}_{n\geq 1}$ converges in $L^2(\Omega)$ to zero, and such
that \begin{equation*} \sum_{n\geq2} \frac{1}{n\log^2
n}\sum_{k=1}^n\frac{1}{k}\E|R_k|^2<\infty
\end{equation*} was established by Nourdin and Peccati in \cite{NP}.
\end{remark}
The proofs of Theorem \ref{ASCLT quotient} and Theorem \ref{ASCLT
sum} are respectively  direct consequences of the two following lemmas:
\begin{lemma}\label{key1
lemma} Let $\{G_{n}\}_{n\geq 1}$ and  $\{R_{n}\}_{n\geq 1}$ be two
sequences of real-valued
 random variables. Define
\begin{eqnarray}
U_{n,\varepsilon}&:=&\left|\frac{1}{\log n}\sum_{k=1}^{n}\frac{1}{k}\ind{G_k
\leq z(1-\varepsilon)}-\phi(z(1-\varepsilon)) \right|,\\
V_{n,\varepsilon}&:=&\left|\frac{1}{\log
n}\sum_{k=1}^{n}\frac{1}{k}\ind{G_k\leq
z(1+\varepsilon)}-\phi(z(1+\varepsilon)) \right|.
\end{eqnarray} Then, for all $z\in\R$,
 $\varepsilon>0$
 \begin{eqnarray*}
  \left|\frac{1}{\log
n}\sum_{k=1}^{n}\frac{1}{k}\ind{G_k\leq z{R_k}}-\phi(z)
\right|&\leq&\max(U_{n,\varepsilon},
V_{n,\varepsilon})+\frac{1}{\log n}\sum_{k=1}^{n}\frac{1}{k}\ind{|R_k-1|\geq\varepsilon}+\varepsilon.
 \end{eqnarray*}
 \end{lemma}
 \begin{lemma}\label{key2 lemma}
Let $\{S_{n}\}_{n\geq 1}$ and  $\{R_{n}\}_{n\geq 1}$ be two
sequences of real-valued
 random variables.  Define
\begin{eqnarray}
T_{n,\eta}&:=&\left|\frac{1}{\log n}\sum_{k=1}^{n}\frac{1}{k}{\bf
1}_{\{ {G_k} \leq z+\eta\}}-\phi(z+\eta) \right|,\\
W_{n,\eta}&:=&\left|\frac{1}{\log
n}\sum_{k=1}^{n}\frac{1}{k}\ind{G_k\leq
z-\eta}-\phi(z-\eta) \right|.
\end{eqnarray}  Then, for all for all $z\in\R$ and $\eta>0$
 \begin{eqnarray*}
  \left|\frac{1}{\log
n}\sum_{k=1}^{n}\frac{1}{k}\ind{G_k+R_k\leq z}-\phi(z)
\right|&\leq&\max(T_{n,\eta}, W_{n,\eta})+\frac{1}{\log
n}\sum_{k=1}^{n}\frac{1}{k}\ind{|R_k|>\eta}+\frac{\eta}{\sqrt{2\pi}}.
 \end{eqnarray*}
 \end{lemma}
 \begin{proof}[Proof of Lemma \ref{key1 lemma}]
 It is inspired from Michael and Pfanzagl \cite[Lemma 1, Page
 78]{MP}. The case $\varepsilon \geq1$ is easy. We now
 assume that $\varepsilon\in(0,1)$.\\
 When $z\geq0$, using the inclusion $\{{G_k}\leq(1-\varepsilon)z\}\subset\{{G_k}\leq zR_k\}\cup
\{ R_k\leq 1-\varepsilon\}$ it comes
\begin{eqnarray}\label{equa1indicator}\ind{G_k\leq
z(1-\varepsilon)}\leq \ind{G_k\leq z R_k} + \ind{
|R_k-1|\geq\varepsilon}. \end{eqnarray}
Since for every
$x\geq0$, $xe^{\frac{-x^2}{2}}\leq e^{\frac{-1}{2}}$, one gets
\begin{eqnarray}\left|\phi(z)-\phi(z(1-\ep))\right|&\leq& \min \left(\frac12,
\frac{z\varepsilon}{\sqrt{2\pi}}e^{\frac{-z^2(1-\ep)^2}{2}}\right)\nonumber\\&\leq&\ep\label{equa1phi}.
\end{eqnarray}
Combining (\ref{equa1indicator}) and (\ref{equa1phi}), one has
 \begin{eqnarray*}\frac{1}{\log
n}\sum_{k=1}^{n}\frac{1}{k}\ind{G_k \leq
z R_k}-\phi(z)&\geq& -U_{n,\varepsilon}-\frac{1}{\log
n}\sum_{k=1}^{n}\frac{1}{k}\ind{|R_k-1|\geq\varepsilon
}-\varepsilon.
\end{eqnarray*}
Now when $z\leq0$, the inclusion $\{{G_k}\leq(1+\varepsilon)z\}\subset\{{G_k}\leq zR_k\}\cup
\{ R_k\geq1+\varepsilon\}$ leads to
\[\ind{G_k
\leq z(1+\varepsilon)}\leq \ind{G_k\leq z R_k} + \ind{|R_k-1|\geq\varepsilon}.\] Moreover since
\begin{eqnarray*}\left|\phi(z)-\phi(z(1+\ep))\right|&\leq&
\frac{|z|\varepsilon}{\sqrt{2\pi}}e^{\frac{-z^2(1+\ep)^2}{2}}\\&\leq&\ep,
\end{eqnarray*}
it comes
 \begin{eqnarray*}\frac{1}{\log
n}\sum_{k=1}^{n}\frac{1}{k}\ind{G_k\leq
z R_k}-\phi(z)&\geq& - V_{n,\varepsilon}-\frac{1}{\log
n}\sum_{k=1}^{n}\frac{1}{k}\ind{|R_k-1|\geq\varepsilon
}-\varepsilon.
\end{eqnarray*}
Thus, for every $z\in\R$
\begin{eqnarray*}\frac{1}{\log
n}\sum_{k=1}^{n}\frac{1}{k}\ind{G_k\leq
z R_k}-\phi(z)&\geq& -\max(U_{n,\varepsilon},
V_{n,\varepsilon})-\frac{1}{\log n}\sum_{k=1}^{n}\frac{1}{k}\ind{|R_k-1|\geq\varepsilon}-\varepsilon.
\end{eqnarray*}
Following the same guidelines as above and using
\begin{eqnarray*}
\mbox{for }z\geq0 \quad \{{G_k}\leq zR_k\}\subset\{{G_k}\leq(1+\varepsilon)z\}\cup \{
R_k\geq1+\varepsilon\}, \\
\mbox{for }z\leq0 \quad \{{G_k}\leq zR_k\}\subset\{{G_k}\leq (1-\varepsilon)z\}\cup
\{R_k\leq 1-\varepsilon\} ,
\end{eqnarray*}
one gets, for every $z\in\R$
\begin{eqnarray*}\frac{1}{\log
n}\sum_{k=1}^{n}\frac{1}{k}\ind{G_k \leq
z R_k}-\phi(z)&\leq& \max(U_{n,\varepsilon},
V_{n,\varepsilon})+\frac{1}{\log n}\sum_{k=1}^{n}\frac{1}{k}\ind{|R_k-1|\geq\varepsilon}+\varepsilon.
\end{eqnarray*}
This completes the proof of Lemma \ref{key1 lemma}.
\end{proof}\\
 \begin{proof}[Proof of Lemma \ref{key2 lemma}] Fix $z\in\R$ and
  $\eta>0$. Remark that   \[\{{G_k}+R_k\leq z\}\subset\{{G_k}\leq z+\eta\}\cup
\{ |R_k|>\eta\}.\]
Thus it comes
\begin{eqnarray*}
   &&\frac{1}{\log
n}\sum_{k=1}^{n}\frac{1}{k}\ind{G_k+R_k\leq z}-\phi(z)
 \\&\leq& \frac{1}{\log n}\sum_{k=1}^{n}\frac{1}{k}\ind{G_k \leq z+\eta}-\phi(z+\eta)+\frac{1}{\log
n}\sum_{k=1}^{n}\frac{1}{k}\ind{|R_k|>\eta}+\phi(z+\eta)-\phi(z)\\&\leq&T_{n,\eta}+\frac{1}{\log
n}\sum_{k=1}^{n}\frac{1}{k}\ind{|R_k|>\eta}+\frac{\eta}{\sqrt{2\pi}}.
 \end{eqnarray*}
 On the other hand, it follows from the inclusion \[\{G_k \leq z-\eta\}\subset\{{G_k}+R_k\leq z \}\cup
\{ |R_k|>\eta\}\] that
\begin{eqnarray*}
   &&\frac{1}{\log
n}\sum_{k=1}^{n}\frac{1}{k}\ind{G_k+R_k\leq z}-\phi(z)\\
 &\geq& \frac{1}{\log n}\sum_{k=1}^{n}\frac{1}{k}\ind{G_k \leq z+\eta}-\phi(z-\eta)-\frac{1}{\log
n}\sum_{k=1}^{n}\frac{1}{k}\ind{|R_k|>\eta}+\phi(z-\eta)-\phi(z)\\&\geq&-W_{n,\eta}-\frac{1}{\log
n}\sum_{k=1}^{n}\frac{1}{k}\ind{|R_k|>\eta
}-\frac{\eta}{\sqrt{2\pi}}.
 \end{eqnarray*}
 The desired conclusion follows.\end{proof}

\section{Application to LSE for fractional Ornstein-Uhlenbeck process}
We first recall  a result of \cite{BNT} concerning the ASCLT for
multiple stochastic integrals.
\begin{theorem}[Bercu \emph{et al.\@} \cite{BNT}]\label{ThemBercu}Let $q \geq 2$ be an integer, and let $\{G_n\}_{n\geq1}$  be a
sequence of the form $G_n=I_q(f_n)$, with $f_n\in{{\cal{H}}^{\odot
q}}$. Assume that $\E[G^2_n] = q!\|f_n\|^2_{{\cal{H}}^{\otimes q}}  =
1$ for all $n$, and  that $G_n$ converges in distribution towards a standard gaussian. Moreover assuming
\begin{eqnarray}
        &&\sum_{n=2}^\infty \frac{1}{n \log^2n}\sum_{k=1}^n\frac{1}{k}\|f_k\otimes_r f_k\|_{{\cal{H}}^{\otimes 2(q-r)}} <
        \infty\ \mbox{ for every } 1 \leq r \leq q-1,\label{firstcondition ofBNT}\\ &&\sum_{n=2}^\infty \frac{1}{n
\log^3n}\sum_{k,l=1}^n\frac{\left|\langle
f_k,f_l\rangle_{{\cal{H}}^{\otimes q}}\right|}{kl}<
       \infty,\label{secondcondition ofBNT}
\end{eqnarray}
then $\{G_n\}_{n\geq1}$ satisfies an ASCLT. In other words, almost
surely, for all $z \in\R$,
\begin{eqnarray*}
  \frac{1}{\log
n}\sum_{k=1}^{n}\frac{1}{k}\ind{G_k\leq z}\limite{n \to \infty}{a.s.}
\phi(z),
 \end{eqnarray*}
or equivalently, almost surely, for any bounded and continuous
function $\varphi:\R\rightarrow\R,$
\[\frac{1}{\log(n)}\sum_{k=1}^{n}\frac{1}{k}\varphi(G_k)\limite{n \to \infty}{a.s.}\E\varphi(N).\]
\end{theorem}
\subsection{Continuous case}In this subsection we apply Theorem
\ref{ASCLT quotient} to  a least squares
estimator for fractional Ornstein-Uhlenbeck processes based on continuous-time observations.\\
Let us consider the fractional Ornstein-Uhlenbeck process $X=\left\{X_t,
t\geq0\right\}$ given by the following linear stochastic
differential equation
\begin{eqnarray}\label{OU}X_0=0,\mbox{ and }\quad  dX_t=-\theta X_tdt+dB_t,\quad t\geq0,
\end{eqnarray}where   $B=\left\{B_t,
t\geq0\right\}$ is a fractional Brownian motion of Hurst index $H\in(\frac{1}{2},1)$ and $\theta$ is a real unknown parameter.
\\
Let $\widehat{\theta}_t$ be a least squares estimator (LSE) of
$\theta$, that is given by
\begin{eqnarray}\label{(LSE)} \widehat{\theta}_t =\frac{\int_0^tX_s\ d X_s}{\int_0^tX_s^2ds},\quad t>0.
\end{eqnarray}This LSE is obtained by the least squares technique, that is $ \widehat{\theta}_t$ (formally) minimizes
\[\theta \longmapsto \int_0^t\left|\dot{{X}}_s+\theta X_s\right|^2ds.\]
The linear equation (\ref{OU}) has the following explicit solution:
\begin{eqnarray}\label{explicit solution}X_t=e^{\theta t}\int_0^te^{-\theta s} d B_s,\qquad t>0,
\end{eqnarray}
Using the equation (\ref{OU}) and (\ref{explicit solution}) the LSE $\{\widehat{\theta}_t\}$ defined in (\ref{(LSE)}) can be written as
follows:
\begin{eqnarray}
\widehat{\theta}_t-\theta=-\frac{\int_0^tX_s d
B_s}{\int_0^tX_s^2ds} =-\frac{\int_0^td
B_se^{\theta s}\int_0^sd B_r e^{-\theta r} }{\int_0^tX_s^2ds}.
\end{eqnarray}
Thus, one has
\begin{eqnarray}\label{representation of LSE}
\sqrt{t}(\theta-\widehat{\theta}_t)=\frac{F_t}{\frac{1}{t}\int_0^tX_s^2ds},\quad
t>0,
\end{eqnarray}
where $$F_t:=I_2(f_t)$$ is a multiple integral of $f_t$  with
\[f_t(u,v) =\frac{1}{2\sqrt{t}}e^{-\theta|u-v|}\ind{[0,t]}^{\otimes2}(u,v).\]
Until the end of this paper we will use the following notation for all $t>0,$
\begin{eqnarray}\label{defSigma}\sigma_t=\lambda(\theta,H)\sqrt{\E(F_t^2)},\quad \mbox{with }\lambda(\theta,H):=\theta^{-2H}H\Gamma(2H).\end{eqnarray}

 We are now ready to
state the main result of this subsection. We first recall some
results of Hu and Nualart \cite{HN} needed throughout the paper:
\begin{eqnarray}\label{limit of F_T^2}
\E(F_{t}^2)\limite{t \to \infty}{} A(\theta,H),
\end{eqnarray}where $$ A(\theta,H)=\theta^{1-4H}\left(H^2(4H-1)
\left[\Gamma(2H)^2+\frac{\Gamma(2H)\Gamma(3-4H)\Gamma(4H-1)}{\Gamma(2-2H)}\right]\right).
$$Moreover,  for every $t\geq0$
\begin{eqnarray}\label{majoratioofDF-|DF|}
 \E\left[\left(\|DF_{t}\|_{{\cal{H}}}^2-\E\|DF_{t}\|_{{\cal{H}}}^2\right)^2\right]
\leq C(\theta,H) t^{8H-6},
\end{eqnarray}and as $t\rightarrow\infty$\begin{eqnarray}\label{limit in distribution of
F_T^2}F_t\overset{\rm distribution}{\longrightarrow} N\sim \mathcal{N }(0,
A(\theta,H)).\end{eqnarray}
At last, one has the convergence
\begin{equation}\label{convergence of dominator}
\frac{1}{t}\int_{0}^{t}X_{s}^{2}ds\limite{t \to \infty}{a.s.} \lambda(\theta,H).
\end{equation}
\begin{theorem}\label{ASCLT continuous case}Assume  $H\in(1/2,3/4)$.
  Then, almost surely,
for all $z\in\R$,
\[ \frac{1}{\log n}\sum_{k=1}^{n}\frac{1}{k}\ind{\frac{\sqrt{k}}{\sigma_k}(\theta-\widehat{\theta}_k)\leq
z}\limite{n \to \infty}{}\phi(z)\] or,
equivalently, almost surely, for any bounded and continuous function
$\varphi$
\[ \frac{1}{\log n}\sum_{k=1}^{n}\frac{1}{k}\varphi\left(\frac{\sqrt{k}}{\sigma_k}(\theta-\widehat{\theta}_k)\right)
\limite{n \to \infty}{} \E(\varphi(N)).\]
\end{theorem}
\begin{proof}
Let us consider, for each $t >0$,
\[G_t=\frac{1}{\sqrt{\E(F_t^2)}}F_{t}=\frac{1}{\sqrt{\E(F_t^2)}}I_2(f_t)\]
and\[R_t=\frac{1}{\lambda(\theta,H) {t}}\int_0^{t}X_s^2ds.\]
 Thus, (\ref{representation of LSE}) leads to
\[\frac{\sqrt{n}}{\sigma_n }(\theta-\widehat{\theta}_n)=G_n/R_n, \quad n\geq1.\]
It follows from  (\ref{convergence of dominator}) that $R_n$ converges almost surely to one as $n$ tends to $\infty$. Then,
 using Theorem~\ref{ASCLT quotient} it suffices to show
 that $\{{G_n} \}_{n\geq1}$ satisfies the ASCLT. To do that, it is sufficient to prove that $\{{G_n} \}_{n\geq1}$
satisfies the conditions of Theorem~\ref{ThemBercu}.
\\
The convergence of $G_n$ towards the standard Gaussian is a straightforward consequence of (\ref{limit
of F_T^2}) and (\ref{limit in distribution of F_T^2}). It remains to fulfill the conditions
(\ref{firstcondition ofBNT}) and (\ref{secondcondition ofBNT}).
Hence, we shall prove that
\begin{eqnarray}\label{first inequality}I=\sum_{n\geq2}\frac{1}{n\log^2(n)}\sum_{k=1}^n\frac{1}{k}\|f_k\otimes_1f_k\|_{\cal{H}^{\otimes2}}<\infty,
\end{eqnarray}
and
\begin{eqnarray}\label{second inequality}J=\sum_{n\geq2}\frac{1}{n\log^3(n)}\sum_{k,l=1}^n\frac{|<f_k,f_l>_{\cal{H}}|}{kl} <\infty.
\end{eqnarray}
Let us deal with the first convergence (\ref{first inequality}). For every
$t>0$, one has
\begin{eqnarray}\label{estimateofDF-|DF|}
\E\left[\left(\|DF_{t}\|_{{\cal{H}}}^2-\E\|DF_{t}\|_{{\cal{H}}}^2\right)^2\right]
=16\|f_t\otimes_1f_t\|^2_{\cal{H}^{\otimes2}}.
\end{eqnarray}
Combining (\ref{majoratioofDF-|DF|}) and (\ref{estimateofDF-|DF|}) it comes
\begin{eqnarray}I \leq C(\theta,H)
\sum_{n\geq2}\frac{1}{n\log^2(n)}\sum_{k=1}^n\frac{1}{k^{4-4H}}<\infty,
\end{eqnarray} since $H<3/4$.\\
Now, we prove (\ref{second inequality}). Let $k<l$, then for some
$k^*\in[0,k]$ we have
\begin{eqnarray*}|<f_k,f_l>_{\cal{H}}|&=&H^2(2H-1)^2\frac{1}{\sqrt{kl}}\int_{[0,k]^2}dxdu\ e^{-\theta|x-u|}\int_{[0,l]^2}dydv\
 e^{-\theta|y-v|}|x-y|^{2H-2}|u-v|^{2H-2}\\
 &=&2H^2(2H-1)^2\sqrt{\frac{k}{l}}\int_{[0,k^*]}du\ e^{-\theta|k^*-u|}\int_{[0,l]^2}dydv\
 e^{-\theta|y-v|}|k^*-y|^{2H-2}|u-v|^{2H-2}\\&:=&2H^2(2H-1)^2\sqrt{\frac{k}{l}}\left(D^{(1)} +D^{(2)} +D^{(3)} +D^{(4)} \right).
\end{eqnarray*}
Moreover, the first term can be bounded above by
\begin{eqnarray*} D^{(1)}&=& \int_{[0,k^*]}du\ e^{-\theta(k^*-u)}\int_{[0,k^*]^2}dydv\
 e^{-\theta|y-v|}(k^*-y)^{2H-2}|u-v|^{2H-2}\\&=&\int_{[0,k^*]^3}\ e^{-\theta u} \
 e^{-\theta|y-v|}y^{2H-2}|u-v|^{2H-2}dudvdy\\
 &\leq&\int_{[0,\infty)^3}\ e^{-\theta u} \
 e^{-\theta|y-v|}y^{2H-2}|u-v|^{2H-2}dudvdy<\infty.
\end{eqnarray*}The last inequality is a consequence of \cite[Proof of Lemma 5.3 of web-only Appendix]{HN}.\\
Following the same guidelines, one gets for the other terms
\begin{eqnarray*} D^{(2)}&=& \int_{[0,k^*]}du\ e^{-\theta(k^*-u)}\int_{[k^*,l]^2}dydv\
 e^{-\theta|y-v|}(y-k^*)^{2H-2}|u-v|^{2H-2}\\&=&\int_{[0,k^*]}du\ e^{-\theta u}\int_{[0,l-k^*]^2}dydv\
 e^{-\theta|y-v|}y^{2H-2}(u+v)^{2H-2}\\
 &\leq&\int_{[0,\infty)^3}\ e^{-\theta u} \
 e^{-\theta|y-v|}y^{2H-2}|u-v|^{2H-2}dudvdy<\infty,
\end{eqnarray*}
\begin{eqnarray*} D^{(3)}&=& \int_{[0,k^*]}du\ e^{-\theta(k^*-u)}\int_{[0,k^*]}dy\int_{[k^*,l]}dv\
 e^{-\theta|y-v|}(k^*-y)^{2H-2}|u-v|^{2H-2}\\&=&\int_{[0,k^*]}du\ e^{-\theta u}\int_{[0,k^*]}dy\int_{[0,l-k^*]}dv\
 e^{-\theta(y+v)}y^{2H-2}(u+v)^{2H-2}\\
 &\leq&\int_{[0,\infty)^3}\ e^{-\theta u} \
 e^{-\theta|y-v|}y^{2H-2}|u-v|^{2H-2}dudvdy<\infty,
\end{eqnarray*} and
\begin{eqnarray*} D^{(4)}&=& \int_{[0,k^*]}du\ e^{-\theta(k^*-u)}\int_{[k^*,l]}dy\int_{[0,k^*]}dv\
 e^{-\theta|y-v|}(y-k^*)^{2H-2}|u-v|^{2H-2}\\&=&\int_{[0,k^*]}du\ e^{-\theta u}\int_{[0,l-k^*]}dy\int_{[0,k^*]}dv\
 e^{-\theta(y+v)}y^{2H-2}|u-v|^{2H-2}\\
 &\leq&\int_{[0,\infty)^3}\ e^{-\theta u} \
 e^{-\theta|y-v|}y^{2H-2}|u-v|^{2H-2}dudvdy<\infty.
\end{eqnarray*}
Thus, we deduce that, for every $k<l$
\begin{eqnarray*}|<f_k,f_l>_{\cal{H}}|&=&C(\theta,H)
\sqrt{\frac{k}{l}}.
\end{eqnarray*}
Consequently it comes
\begin{eqnarray}J&\leq& C(\theta,H) \sum_{n\geq2}\frac{1}{n\log^3(n)}\sum_{l=1}^n\frac{1}{l^{3/2}} \sum_{k=1}^l\frac{1}{\sqrt{k}}\nonumber\\
&\leq& C(\theta,H)
\sum_{n\geq2}\frac{1}{n\log^3(n)}\sum_{l=1}^n\frac{1}{l}\nonumber\\
&\leq& C(\theta,H) \sum_{n\geq2}\frac{1}{n\log^2(n)}<\infty,\label{J
bound}
\end{eqnarray}
which concludes the proof.
\end{proof}
\subsection{Discrete case}
Consider the fractional Ornstein-Uhlenbeck process $X=\left\{X_t,
t\geq0\right\}$ defined in (\ref{OU}). Assume that the process $X$
is observed equidistantly in time with the step size $\Delta_n$:
$t_i=i\Delta_{n}, i=0,\ldots,n$, and $T_n=n\Delta_{n}$ denotes the
length of the `observation window'.
\\Let $\widetilde{\theta}_{n }$ be a least squares estimator defined  as
follows: $\widetilde{\theta}_{n}$ minimizes
\[
\theta\mapsto \sum_{i=1}^{n}\left|X_{t_i}-X_{t_{i-1}}+\theta
X_{t_{i-1}}\Delta_{n}\right|^2.
\] Thus
$\widetilde{\theta}_n$ is given by
\begin{equation}\label{estimator}
\widetilde{\theta}_{n}=-\frac{\sum_{i=1}^{n}
X_{t_{i-1}}(X_{t_{i}}-X_{t_{i-1}})}{\Delta_{n}\sum_{i=1}^{n}
X_{t_{i-1}}^2}.
\end{equation}
Using (\ref{OU}), one has
\begin{eqnarray}\label{representation estimator-theta}
\widetilde{\theta}_{n}-\theta&=&-\frac{\sum_{i=1}^{n}
X_{t_{i-1}}U_i}{\Delta_{n}\sum_{i=1}^{n} X_{t_{i-1}}^2}
\end{eqnarray}where
\begin{eqnarray*}\label{expression of U}
U_i=-\theta\int_{t_{i-1}}^{t_{i}}X_sds+B_{t_{i}}-B_{t_{i-1}},\qquad
 i=1,\ldots,n.
\end{eqnarray*}
In this subsection we first prove the strong consistency of
$\widetilde{\theta}_{n}$. Then, applying Theorem~\ref{ASCLT
quotient} and Theorem~\ref{ASCLT sum},  we prove the ASCLT for the
LSE $\widetilde{\theta}_n$.

For the strong consistency,  let
us state the following  direct consequence of the Borel-Cantelli
Lemma (see e.g.\@ \cite{KN}), which allows us to turn convergence
rates in the $p$-th mean into pathwise convergence rates.
\begin{lemma}\label{pathwise convergence}
Let $\gamma >0$ and $p_0 \in \mathbb{N}$. Moreover let $(Z_{n})_{n \in \N}$
be a sequence of random variables. If  for every $p \geq p_0$ there
exists a constant $c_p>0$  such that for all $n \in \N$,
$$(\E |Z_{n}|^{p})^{1/p} \leq c_p \cdot n^{-\gamma},$$  then
for all $\varepsilon > 0$ there  exists a random variable
$\eta_{\varepsilon}$ such that
$$ |Z_{n}| \leq \eta_{\varepsilon} \cdot n^{-\gamma + \varepsilon}
\quad \mbox{almost surely} $$ for all $n \in \N$.  Moreover,   $\E
|\eta_{\varepsilon}|^{p} < \infty$ for all $p \geq 1$.
\end{lemma}

From now on, assume that $\Delta_n=t_{k+1}-t_{k}=
n^{-\alpha}$ with a given $\alpha \in (\frac{1}{2H+1},1)$.

Let us now prove the strong consistency of $\widetilde{\theta}_n$.
\begin{theorem}\label{strong consistency}Assume $H\in(1/2,1)$. Then, if $\Delta_{n} \rightarrow 0$ and $n\Delta_{n}\rightarrow
\infty$  as $n\rightarrow \infty$,  \begin{eqnarray}\widetilde{\theta}_n \limite{n \to \infty}{a.s.}\theta.
\end{eqnarray}
\end{theorem}
\begin{proof}
 We first prove that
\[\frac{1}{n}\sum_{i=1}^{n} X_{t_{i-1}}^2\]
converges almost surely to $\lambda(\theta,H)$ as $n\rightarrow\infty$. Using (\ref{convergence of dominator}) it suffices to show that \begin{equation} \frac{1}{n}\sum_{i=1}^{n}
X_{t_{i-1}}^2-\frac{1}{T_n}\int_{0}^{T_n}X_{t}^{2}dt:=\frac{1}{T_n}\int_{0}^{T_n}Z_n(t)\
dt\limite{n \to \infty}{a.s.}0,
\end{equation}
where
$$ Z_n(t)= X_{t_{i-1}}^2 - X_{t}^{2} , \qquad t \in [t_{i-1},t_{i}), \qquad i=1,
\ldots,n.
$$
Moreover, it is straightforward (see \cite{GKN}, or \cite{khalifa}) to check that :
for any   $p\geq 1$  there exist constants $c_p, C_p
>0$ such that
\begin{equation}\label{moments of X}
    \E(|X_t|^{p})  \leq c_p, \mbox{ and } \  \E(|X_t-X_s|^{p})  \leq C_p
|t-s|^{pH},
 \qquad\mbox{for all}\quad s,t\ge 0.
\end{equation}
Consequently, one has
\[\E\left(|Z_n(t)|^p\right)\leq C(\theta,H) \Delta_n^{pH}.\]
Applying Hölder inequality it comes
\[\E\left(\left|\frac{1}{T_n}\int_{0}^{T_n}Z_n(t)\ dt\right|^p\right)\leq C(\theta,H)  \Delta_n^{pH}.\]
Then Lemma~\ref{pathwise convergence} yields
\begin{eqnarray}\label{integral of Z_n} \frac{1}{T_n}\int_{0}^{T_n}Z_n(t)\ dt \limite{n \to \infty}{a.s.}0.\  \end{eqnarray}
Combining (\ref{integral of Z_n}) and (\ref{convergence of
dominator}) one gets
\begin{eqnarray}\frac{1}{n}\sum_{i=1}^{n}
X_{t_{i-1}}^2\limite{n \to \infty}{a.s.}\lambda(\theta,H).\label{convergenceDiscretDenominator}\end{eqnarray}
Clearly, (\ref{representation estimator-theta}) implies
\begin{eqnarray}\label{representationDiscretLSE} \theta-\widetilde{\theta}_{n}&=&
\frac{\frac{1}{T_n}\sum_{i=1}^{n}
X_{t_{i-1}}U_i}{\frac{1}{n}\sum_{i=1}^{n} X_{t_{i-1}}^2}.
\end{eqnarray}In order to prove Theorem \ref{strong consistency},  using (\ref{representationDiscretLSE})
 and (\ref{convergenceDiscretDenominator}), it is sufficient  to show that
\begin{eqnarray}  \frac{1}{T_n}\sum_{i=1}^{n}
X_{t_{i-1}}U_i\limite{n \to \infty}{a.s.}0.
\end{eqnarray}
Use the following decomposition:
\begin{eqnarray}&& \frac{1}{T_n}\sum_{i=1}^{n}
X_{t_{i-1}}U_i\nonumber\\&=& \frac{1}{T_{n}}\sum_{i=1}^{n}
X_{t_{i-1}}[U_i-(B_{t_i}-B_{t_{i-1}})]+\frac{1}{T_{n}}
\sum_{i=1}^{n}\int_{t_{i-1}}^{t_{i}}(X_{t_{i-1}}-X_t)dB_t+\frac{1}{T_{n}}\int_0^{T_n}
X_t dB_t\nonumber\\&:=&J_1(n)+J_2(n)+J_3(n).\label{decompositionNumerator}
\end{eqnarray}
We first study $J_1(n)$. Fix $p\geq 1$
\begin{eqnarray*}\E(|J_1(n)|^p)&=& \E\left(\left|\frac{1}{T_{n}}\sum_{i=1}^{n}
X_{t_{i-1}}[U_i-(B_{t_i}-B_{t_{i-1}})]\right|^p\right)\\&=&\E\left(\left|\frac{1}{T_{n}}
\sum_{i=1}^{n} X_{t_{i-1}}
\int_{t_{i-1}}^{{t_i}}\theta(X_{t_{i-1}}-X_t)dt \right|^p\right)\\&
=&\E\left(\left|\frac{1}{T_{n}}
 \int_{0}^{T_n}Z_{1,n}(t)dt \right|^p\right),
\end{eqnarray*}where
$$ Z_{1,n}(t)=\theta  X_{t_{i-1}}(X_{t_{i-1}} - X_{t}) , \qquad t \in [t_{i-1},t_{i}), \qquad i=1,
\ldots,n.
$$
Using (\ref{moments of X}) together  with Cauchy
inequality there exists a
  constant $c$ depending only on $p,\  \theta$ and $H$ such that
\[\E\left(|Z_{1,n}(t)|^p\right)\leq c \Delta_n^{pH}.\]
Applying Hölder inequality it comes
\begin{eqnarray}\label{majorationofJ_1}\E(|J_1(n)|^p)=\E\left(\left|\frac{1}{T_n}\int_{0}^{T_n}Z_{1,n}(t)\right|^p\right)\leq c  \Delta_n^{pH}.\end{eqnarray}
According to Lemma \ref{pathwise convergence} we obtain
\begin{eqnarray} J_1(n) \limite{n \to \infty}{a.s.}0.\end{eqnarray}
Using the same argument as in the proof of \cite[Theorem
3.3]{khalifa}, (\ref{majoration second moment of
skorohod}) leads to
\begin{eqnarray} \E(|J_2(n)|^p)&\leq& c
n^{p(H-1)}\Delta_n^{p(2H-1)}.\label{majorationofJ_2}\end{eqnarray}
Thus, Lemma \ref{pathwise convergence} yields
 \begin{eqnarray*} J_2(n) \limite{n \to \infty}{a.s.}0.\end{eqnarray*}
 Finally, by (\ref{limit of F_T^2}) we have
 \begin{eqnarray*} \E(|J_3(n)|^p)&\leq& c T_n^{-p/2}\\&\leq& c
 n^{-\frac{p}{2}(1-\alpha)}.
 \end{eqnarray*}
Applying Lemma~\ref{pathwise convergence}, the sequence $(J_3(n))_n$
converges almost surely to zero which completes the proof.
\end{proof}

\begin{theorem}\label{ASCLT
discrete case}Assume $H\in(1/2,3/4)$. Then almost surely, for all
$z\in\R$,
\[ \frac{1}{\log n}\sum_{k=1}^{n}\frac{1}{k}{\bf
1}_{\{\frac{\sqrt{T_k}}{\sigma_{T_k}}(\theta-\widetilde{\theta}_k)\leq
z\}}\limite{n \to \infty}{a.s.}\phi(z),\] or
equivalently, for any bounded and continuous function
$\varphi$
\[ \frac{1}{\log n}\sum_{k=1}^{n}\frac{1}{k}\varphi\left(\frac{\sqrt{T_k}}{\sigma_{T_k}}(\theta-\widetilde{\theta}_k)\right)
\limite{n\to \infty}{a.s.} \E(\varphi(N)).\]
\end{theorem}
\begin{proof}
Let us define
\begin{eqnarray*}\label{representationofG_nbar}
\overline{G}_n=\frac{1}{\sqrt{T_nE(F_{T_n}^2)}}\sum_{i=1}^{n}
X_{t_{i-1}}U_i
\end{eqnarray*}
and
\begin{eqnarray*}\label{representationofR_nbar} \overline{R}_n={\frac{1}{n\lambda(\theta,H)}\sum_{i=1}^{n}
X_{t_{i-1}}^2}
\end{eqnarray*} such that, from (\ref{representationDiscretLSE}), the convergence rate of the estimator can be
rewritten as:
\begin{eqnarray*}\label{representationNormalizedDiscretLSE} \frac{\sqrt{T_n}}{\sigma_{T_n}}(\theta-\widetilde{\theta}_n)&=&
\overline{G}_n/\overline{R}_n.
\end{eqnarray*}
Clearly, (\ref{convergenceDiscretDenominator}) yields
\begin{eqnarray*} \overline{R}_n \limite{n \to \infty}{a.s.} 1.\end{eqnarray*}
Hence,   using Theorem \ref{ASCLT quotient} it remains to prove
that $\{\overline{G}_n\}_{n\geq1}$ satisfies the ASCLT.\\
It follows from (\ref{decompositionNumerator}) that
\begin{eqnarray*}
\overline{G}_n=\sqrt{\frac{T_n}{\E(F_{T_n}^2)}}\left(J_1(n)+J_2(n)+J_3(n)\right).
\end{eqnarray*}
Combining (\ref{majorationofJ_1}), (\ref{limit of F_T^2}),
(\ref{majorationofJ_2}), $\alpha>\frac{1}{2H+1}$ and Lemma
\ref{pathwise convergence} it comes
\begin{eqnarray*}\label{majorationofJ_1+J_2}
\sqrt{\frac{T_n}{\E(F_{T_n}^2)}}\left(J_1(n)+J_2(n)\right)
\limite{n \to \infty}{a.s.}0.
\end{eqnarray*}
Now, by  Theorem \ref{ASCLT sum}, it remains to prove that
$\left\{\sqrt{\frac{T_n}{\E(F_{T_n}^2)}} J_3(n)\right\}_{n\geq1}$
satisfies the ASCLT.\\
We have \[\sqrt{\frac{T_n}{\E(F_{T_n}^2)}} J_3(n)=G_{T_n}\overset{\rm
distribution}{\longrightarrow} N\sim \mathcal{N }(0, 1)\ \mbox{ as }
n\rightarrow\infty.\] Moreover, using the same arguments as in the
proof of Theorem \ref{ASCLT continuous case},
$\left\{\sqrt{\frac{T_n}{\E(F_{T_n}^2)}} J_3(n)\right\}_{n\geq1}$
satisfies the conditions (\ref{firstcondition ofBNT}) and
(\ref{secondcondition ofBNT}) of Theorem \ref{ThemBercu}. This concludes the proof of Theorem~\ref{ASCLT discrete case}.
\end{proof}

\end{document}